\documentclass[journal,twoside,web]{ieeecolor}
\usepackage{generic}
\usepackage{cite}
\usepackage{amsmath,amssymb,amsfonts}
\usepackage{graphicx}
\usepackage{textcomp}
\usepackage{amsfonts}

\usepackage{subfigure}
\usepackage{subeqnarray,cases}
\usepackage{hyperref}
\usepackage[misc]{ifsym}

\usepackage[justification=centering]{caption}
\usepackage[ruled]{algorithm2e}

\newtheorem{theorem}{Theorem}[section]

\newtheorem{lemma}{Lemma}[section]
\newtheorem{proposition}{Proposition}[section]
\newtheorem{remark}{Remark}[section]

\def\BibTeX{{\rm B\kern-.05em{\sc i\kern-.025em b}\kern-.08em
    T\kern-.1667em\lower.7ex\hbox{E}\kern-.125emX}}

\begin{document}
\title{``Second-Order Primal'' + ``First-Order Dual'' Dynamical Systems with Time Scaling  for Linear Equality Constrained Convex Optimization Problems }

\author{Xin He,  Rong Hu  and  Ya-Ping Fang
\thanks{This work was supported by  the National Natural Science  Foundation of China (11471230).}
\thanks{Xin He is with Department of Mathematics, Sichuan University, Chengdu, Sichuan, P.R. China (e-mail:hexinuser@163.com).}
\thanks{Rong Hu is with Department of Applied Mathematics, Chengdu University of Information Technology, Chengdu, Sichuan, P.R. China (e-mail:ronghumath@aliyun.com).}
\thanks{Ya-Ping Fang is with Department of Mathematics, Sichuan University, Chengdu, Sichuan, P.R. China (e-mail:ypfang@scu.edu.cn. The corresponding author).}
}

\maketitle

\begin{abstract}
Second-order dynamical systems are important tools for solving optimization problems, and most of  existing works in this field have focused on unconstrained optimization problems. In this paper, we propose an inertial primal-dual dynamical system with constant viscous damping and time scaling for the linear equality constrained convex optimization problem, which consists of a second-order ODE for the primal variable and a first-order ODE for the dual variable. When the scaling satisfies certain conditions, we prove its  convergence property without assuming strong convexity. Even the convergence rate can become exponential when the scaling grows exponentially.   We also show that the obtained  convergence property of the dynamical system is preserved under a small perturbation.
\end{abstract}


\section{Introduction}
\IEEEPARstart{A}{}  basic problem arising in many applications such as compressed sensing, statistical estimation, machine learning, the global consensus problem, and image processing, is the linear equality constrained convex optimization problem
	\begin{equation}\label{question}
				\min_{x}  \quad f(x), \quad s.t.  \  Ax = b,
	\end{equation}
where $A\in\mathbb{R}^{m\times n}$, $b\in\mathbb{R}^{m}$, and $f: \mathbb{R}^{n}\to\mathbb{R}\cup\{+\infty\}$ is a  proper, convex and continuously  differentiable function.  See e.g. \cite{BeckA,LinZC,Boyd2010,Chambolle2016,Candes2008,Cao2012,Zargham2013,Liu2017} for some examples. 

Dynamical system approaches have been shown to be efficient tools for solving  optimization problems, and it can  provide more insights into the existing numerical methods. The gradient flow system
$$\dot x(t)=-\nabla f(x(t))$$
can be understood as the continuous limit of  the proximal point algorithm and the gradient method  for  the unconstrained convex  optimization problem
\begin{equation}\label{min_fun}
   \min_x f(x),
\end{equation}
where $f(x)$ is a proper, convex and differentiable function. More precisely, the implicit  and  explicit  schemes  of the gradient flow system lead to the proximal point algorithm and the gradient method respectively.
It was shown by Alvarez \cite{Alvarez2000}  that an implicit discretization scheme of the heavy ball with friction system given by Polyak \cite{Polyak1964}
\begin{equation*}\label{dy:hball}
({HBF})\qquad\ddot{x}(t)+\gamma\dot{x}(t)+\nabla f(x(t))=0,
\end{equation*}
where $\gamma>0$ is a constant damping, can lead to an inertial proximal-like  algorithm for  the problem \eqref{min_fun}.  Su et al. \cite{Su2014} showed that the inertial  system
\begin{equation*}
 ({AVD}_{\alpha})\qquad	\ddot{x}(t)+\frac{\alpha}{t}\dot{x}(t)+\nabla f(x(t))=0
\end{equation*}
with $\alpha=3$ can be understood  as the continuous limit of  the Nesterov's accelerated gradient algorithm \cite{Nesterov1983} and  its widely used successors like the FISTA algorithm of Beck and Teboulle \cite{Beck2009} for  the problem \eqref{min_fun}.  Attouch et al. \cite{AttouchCPR2018} generalized the results of  \cite{Su2014} by considering an additional perturbation.
Without assuming strong convexity of $f$,  Balhag el al. \cite{Balhag2020} derived the exponential convergence of the inertial gradient system with  time scaling $\beta(t)$
\begin{equation}\label{dy_HBFT}
	 \ddot{x}(t)+\gamma\dot{x}(t)+\beta(t)\nabla f(x(t))=0
\end{equation} under the assumption $\beta(t)=e^{\beta t}$ with $\beta\leq \gamma$.   Attouch et al. \cite{AttouchCRF2019}  considered  the following second-order dynamical system
\begin{equation*}
 ({AVD}_{\alpha,\beta})\qquad	\ddot{x}(t)+\frac{\alpha}{t}\dot{x}(t)+\beta(t)\nabla f(x(t))=0
\end{equation*}
for the problem \eqref{min_fun}, and  proved the $\mathcal{O}(1/t^2\beta(t))$ convergence rate under merely convexity assumption of $f$.   The  convergence properties of $({AVD}_{\alpha,\beta})$  were also discussed by Wibisono et al. \cite{WibisonoWJ2016} and Wilson et al. \cite{Wilson2016} from a variational perspective. For more results on dynamical system approaches for unconstrained optimization problems, we refer the reader to \cite{Shi2018,SunY2019,Bolte2003,Bot2016,LuoMP}.

As mentioned earlier, there are many works on inertial dynamical system approaches for unconstrained optimization problems. However, the study on dynamical system methods for constrained optimization problems is still in its infancy, and most of existing works on  dynamical system approaches for constrained optimization problems focused on first-order dynamical systems, (see e.g.\cite{Cherukuri2016,Feijer2010,Zeng2016,Garg2020,Cherukuri2017,Qu2018,Luo2021}). It's worth noting that the  convergence rate analysis of first-order primal-dual dynamical systems always require some additional assumptions, such as strongly convexity \cite{DingJ2019,Tang2020,Wang2021}, metric subregularity \cite{Liang2019}, and PL condition \cite{Hassan2021}. Recently,  some second-order inertial primal--dual dynamical systems have been proposed for linearly constrained optimization problems. We mention some related works as follows.   Zeng et al. \cite{Zeng2019}  proposed a second-order dynamical system based on primal-dual framework and proved  $\mathcal{L}(x(t),\lambda^*)-\mathcal{L}(x^*,\lambda^*)= \mathcal{O}(1/t^2)$, where
$\mathcal{L}(x,\lambda)$ is the Lagrangian function of the problem \eqref{question} and  $(x^*,\lambda^*)$ is a saddle point of  $\mathcal{L}$ in the sense that
\begin{equation}\label{eq_Larg}
	\mathcal{L}(x^*,\lambda)\leq  \mathcal{L}(x^*,\lambda^*)\leq  \mathcal{L}(x,\lambda^*), \qquad \forall (x,\lambda)\in \mathbb{R}^{n}\times\mathbb{R}^m.
\end{equation}
He et al. \cite{HeHF2020} and Attouch et al. \cite{AttouchFC2021} further considered  second-order  primal-dual dynamical systems involving time-dependent positive damping terms for the problem \eqref{question}  with a separable structure, and obtained  some results similar to the ones in  \cite{Zeng2019}. It is worth mentioning that the second-order dynamical systems considered in \cite{Zeng2019,HeHF2020,AttouchFC2021} involve inertial terms  both for the primal variable and the dual variable. Fazlyab et al. \cite{Fazlyab2017}  developed an Euler-Lagrange equation for \eqref{question} and proved that it achieves an exponential convergence under the assumption that $f$ is strongly convex and twice continuously differentiable.

In this paper, we propose a ``second-order primal'' + ``first-order dual'' dynamical system  with constant viscous damping and time scaling,  linked with the Polyak's heavy ball method \cite{Polyak1964}, for the problem \eqref{question}. Under  scaling conditions, we prove the  convergence properties of the proposed dynamical system using  the Lyapunov analysis approaches. Our main contributions are  summarized as follows:
\begin{itemize}
	\item [(a).] We propose a new inertial primal-dual dynamical system with constant viscous damping and time scaling for the problem \eqref{question}.  Compared to the existing inertial primal-dual dynamical systems,  the novelty of  the proposed dynamical system lies in its structure:  a primal-dual dynamical system with a second-order ODE for the primal variable and a first-order ODE for the dual variable.  To the best of our knowledge,  this is the first time to consider  ``second-order primal'' $+$ ``first-order dual'' dynamical systems  for the problem \eqref{question}.
	\item [(b).]  By constructing energy functions, we show that the inertial primal-dual dynamical system achieves an ergodic $\mathcal{O}(1/t)$ convergence rate of the objective function when the scaling coefficient $\beta(t)\equiv\beta>0$, and a nonergodic $\mathcal{O}(1/\beta(t))$ convergence rate of the Lagrangian function when $\lim_{t\to+\infty}\beta(t)=+\infty$.  We show that the convergence rate can become exponential when the scaling grows exponentially.   The convergence properties turn out to be stable under a perturbation $\epsilon(t)$ with $\int^{+\infty}_{t_0}\|\epsilon(t)\|dt<+\infty$. 
\end{itemize}

The rest of this paper is organized as follows: In Section II, we propose  an inertial primal-dual dynamical system with constant viscous damping  and time scaling for the problem \eqref{question}, and investigate the  convergence properties.  In Section III, we consider a perturbed version of the proposed dynamical system, and show that the  convergence properties in Section II still hold with a small perturbation. Finally, we give  concluding remarks in Section IV.

\section{``Second-order primal'' + ``first-order Dual'' dynamical system}

In this section, we  propose  an inertial primal-dual dynamical system with  constant viscous damping and time scaling for the problem \eqref{question}  in terms of the  augmented Lagrangian function, and discuss its   convergence properties. Recall that the  augmented Lagrangian function $\mathcal{L}^{\sigma}(x,\lambda)$ of the  problem \eqref{question} is defined by
\begin{equation*}\label{eq_AugL}
	\mathcal{L}^{\sigma}(x,\lambda) =  \mathcal{L}(x,\lambda)+\frac{\sigma}{2}\|Ax-b\|^2,
\end{equation*}
where
\[  \mathcal{L}(x,\lambda) = f(x)+\langle \lambda,Ax-b\rangle\]
is the Lagrangian function of the problem \eqref{question} and $\sigma > 0$ is the penalty parameter. It is well known that $x^*$ is a solution of  the problem \eqref{question} if and only if there exists $\lambda^*\in \mathbb{R}^m$ such that $(x^*, \lambda^*)$ is a saddle point of  $\mathcal{L}$. It is also known that $(x^*, \lambda^*)$ is a saddle point of  $\mathcal{L}$ if and only if it is a KKT  point of  the  problem \eqref{question} in the sense that
\begin{equation}\label{eq_Saddle_point}
	\begin{cases}
		-A^T\lambda^* =  \nabla f(x^*),\\
		Ax^* -b =0,
	\end{cases}
\end{equation}
where $\nabla f$ is the gradient of $f$.
 Unless  otherwise stated, in this paper  we always assume that the saddle point set $\Omega$ of the Lagrangian function $\mathcal{L}$ is nonempty.   For a fixed $t_0\geq 0$, in terms of the augmented Lagrangian function $\mathcal{L}^{\sigma}$, we propose the following  primal-dual dynamical system,  linked with the Polyak's heavy ball method \cite{Polyak1964}:
\begin{equation*}
	\begin{cases}
		\ddot{x}(t)+\gamma\dot{x}(t)& = -\beta(t)\nabla_x \mathcal{L}^{\sigma}(x(t),\lambda(t)),\\
		 \dot{\lambda}(t) &=\beta(t)\nabla_{\lambda}\mathcal{L}^{\sigma}(x(t)+\delta\dot{x}(t),\lambda(t)),
	\end{cases}
\end{equation*}
where $\gamma, \delta>0$  are two  constant damping  coefficients and $\beta:[t_0,+\infty)\to (0,+\infty)$ is a time scaling coefficient that plays a crucial role in obtaining the  convergence properties. By trivial calculations, we can rewrite it as
\begin{equation}\label{dy_unpertu}
	\begin{cases}
		\ddot{x}(t)+\gamma\dot{x}(t)& = -\beta(t)(\nabla f(x(t))+A^T\lambda(t)\\
		&\quad+\sigma A^T(Ax(t)-b)),\\
		\dot{\lambda}(t)& = \beta(t)(A(x(t)+\delta\dot{x}(t))-b).
	\end{cases}
\end{equation}
It is worth mentioning that the dynamical system \eqref{dy_unpertu} involves the inertial term only for the primal variable, which is very different from the existing primal-dual dynamical systems with inertial terms  both for the primal variable and the dual variable (see \cite{Zeng2019,HeHF2020,AttouchFC2021}). The motivation on the mixed ``second-order''+``first-order'' dynamical system originates from the following observations:  For the acceleration  properties,  the inertial terms are introduced to continuous-time dynamical systems (discrete algorithms) for solving the underlying problems. Also, we notice that the computational cost of a primal-dual algorithm  comes mostly from the subproblem in the primal variable and that  from the point of view of  numerical computations, a first-order ODE is simpler and easier to solve than a second-order ODE in general. A natural problem arises: Whether can we obtain the convergence rates from the literature if we construct the inertial term only for the primal variable? This motivates us to consider the primal-dual dynamical system consisting of a second-order ODE for the primal variable and a first-order ODE for the dual variable for the problem \eqref{question}. In this paper  we will show that  the ``second-order primal''+``first-order dual''  dynamical system \eqref{dy_unpertu} as well as its perturbed version can enjoy  same convergence rates as the ``second-order''+``second-order''  dynamical systems considered in  \cite{Zeng2019,HeHF2020,AttouchFC2021}.

The  existence and uniqueness of local solutions of  \eqref{dy_unpertu} follows directly from  the Picard-Lindelof Theorem (see \cite[Theorem 2.2]{Teschl2012}).

\begin{proposition}\label{local_exist}
	Let $\nabla f$ be locally Lipschitz continuous and  $\beta:[t_0,+\infty)\to (0,+\infty)$ be a continuous function. Then for any $(x_0,\lambda_0,u_0)$, there exists a unique solution $(x(t),\lambda(t))$ with $x(t)\in\mathcal{C}^2([t_0,T),\mathbb{R}^{n})$,  $\lambda(t)\in\mathcal{C}^1([t_0,T),\mathbb{R}^{m})$ of the dynamical system \eqref{dy_unpertu} satisfying $(x(t_0),\lambda(t_0),\dot{x}(t_0))=(x_0,\lambda_0,u_0)$ on a maximal interval $[t_0,T)\subseteq[t_0,+\infty)$.
\end{proposition}

Next, we start to discuss the asymptotic properties of \eqref{dy_unpertu}.

\begin{theorem}\label{th_dy_unpertu}
	 Assume that  $\nabla f$ is locally Lipschitz continuous, $\beta:[t_0,+\infty)\to(0,+\infty)$ is a continuous differentiable function, and the following  scaling condition holds:
\begin{equation}\label{eq_th1_beta_t}
	\dot{\beta}(t)\leq \frac{1}{\delta}\beta(t), \qquad \frac{1}{\delta}< \gamma.
\end{equation}
Then for any $(x_0,\lambda_0,u_0)$, there exists a unique global solution $(x(t),\lambda(t))$ with $x(t)\in\mathcal{C}^2([t_0,+\infty),\mathbb{R}^{n})$, $\lambda(t)\in\mathcal{C}^1([t_0,+\infty),\mathbb{R}^{m})$ of the dynamical system \eqref{dy_unpertu} satisfying $(x(t_0),\lambda(t_0),\dot{x}(t_0))=(x_0,\lambda_0,u_0)$. Moreover, $(x(t),\lambda(t))$ is bounded on $[t_0,+\infty)$, and for any $(x^*,\lambda^*)\in\Omega$  the following conclusions hold:
\begin{itemize}
	\item [(i)]$\int^{+\infty}_{t_0} (\frac{1}{\delta}\beta(t)-\dot{\beta}(t))(\mathcal{L}^{\sigma}(x(t),\lambda^*)-\mathcal{L}^{\sigma}(x^*,\lambda(t))) dt <+\infty$.
	\item [(ii)] $\int^{+\infty}_{t_0}\beta(t)\|Ax(t)-b\|^2 dt <+\infty$,  $\int^{+\infty}_{t_0}\|\dot{x}(t)\|^2dt <+\infty$.
	\item [(iii)] If  $ \lim_{t\to+\infty}\beta(t) = +\infty$, then
	\[ \mathcal{L}(x(t),\lambda^*)-\mathcal{L}(x^*,\lambda(t)) =\mathcal{O}\left(\frac{1}{\beta(t)}\right),\]
	and
	\[ \|Ax(t)-b\| =\mathcal{O}\left(\frac{1}{\sqrt{\beta(t)}}\right).\]
\end{itemize}
\end{theorem}
\begin{proof}
By Proposition \ref{local_exist}, there exists a unique local solution $(x(t),\lambda(t))$ of \eqref{dy_unpertu} defined on a maximal interval $[t_0 , T)$ with $T\leq +\infty$ for any initial value.

We first show  $T = +\infty$.  Fix $(x^*,\lambda^*)\in\Omega$. Consider the energy function $\mathcal{E}:[t_0,T)\to[0,+\infty)$ defined by
\begin{equation}\label{eq_th1_1}
	\mathcal{E}(t) = \mathcal{E}_0(t)+\mathcal{E}_1(t),
	\end{equation}
where
\begin{equation*}\label{eq_th1_2}
	\begin{cases}
		\mathcal{E}_0(t) =\beta(t)(\mathcal{L}^{\sigma}(x(t),\lambda^*)-\mathcal{L}^{\sigma}(x^*,\lambda^*)	),\\
		\mathcal{E}_1(t) = \frac{1}{2}\|\frac{1}{\delta}(x(t)-x^*)+\dot{x}(t)\|^2+\frac{\delta\gamma-1}{2\delta^2}\|x(t)-x^*\|^2\\
	\qquad\qquad+\frac{1}{2\delta}\|\lambda(t)-\lambda^*\|^2.
	\end{cases}
\end{equation*}
Differentiate  $\mathcal{E}_0(t)$ to get
\begin{eqnarray}\label{eq_th1_3}
	&&\dot{\mathcal{E}}_0(t) =\dot{\beta}(t)(\mathcal{L}^{\sigma}(x(t),\lambda^*)-\mathcal{L}^{\sigma}(x^*,\lambda^*))\\
	 &&\qquad+ \beta(t)\langle \nabla f(x(t))+A^T\lambda^*+\sigma A^T(Ax(t)-b),\dot{x}(t)\rangle.\nonumber
\end{eqnarray}
Since $(x^*,\lambda^*)\in\Omega$ and $Ax^*=b$, from \eqref{dy_unpertu} we have
\begin{eqnarray}\label{eq_th1_4}
	&& \dot{\mathcal{E}}_1(t) = \langle
			\frac{1}{\delta}(x(t)-x^*)+\dot{x}(t),\frac{1}{\delta}\dot{x}(t)+\ddot{x}(t)\rangle\nonumber\\
		&&\quad	+ \frac{\delta\gamma-1}{\delta^2}\langle x(t)-x^*,\dot{x}(t)\rangle + \frac{1}{\delta}\langle \lambda(t)-\lambda^*,\dot{\lambda}(t)\rangle\nonumber  \\
		&& =\langle
			\frac{1}{\delta}(x(t)-x^*)+\dot{x}(t),(\frac{1}{\delta}-\gamma)\dot{x}(t)\rangle\nonumber \\
		&&\quad -\beta(t)\langle
			\frac{1}{\delta}(x(t)-x^*)+\dot{x}(t),\nabla f(x(t))+A^T\lambda(t)\rangle\nonumber\\
		&&\quad -\sigma\beta(t)\langle
			\frac{1}{\delta}(x(t)-x^*)+\dot{x}(t), A^T(Ax(t)-b)\rangle\nonumber \\			
		&&\quad +\frac{\delta\gamma-1}{\delta^2}\langle x(t)-x^*,\dot{x}(t)\rangle\\
		&&\quad +\frac{\beta(t)}{\delta}\langle \lambda(t)-\lambda^*,A(x(t)-x^*)+\delta A\dot{x}(t)\rangle\nonumber\\
		&&= \frac{1-\delta\gamma}{\delta}\|\dot{x}(t)\|^2\nonumber\\
		&&\ - \frac{\beta(t)}{\delta}\langle x(t)-x^*, \nabla f(x(t))+ A^T\lambda^*+\sigma A^T(Ax(t)-b)\rangle\nonumber \\
		&&\  - \beta(t)\langle \dot{x}(t), \nabla f(x(t))+A^T\lambda^*+\sigma A^T(Ax(t)-b)\rangle.\nonumber
\end{eqnarray}
Since $f$ is a convex function and $Ax^*=b$,  
\begin{small}
\begin{eqnarray*}
	&&\langle x(t)-x^*, \nabla f(x(t))+ A^T\lambda^*+\sigma A^T(Ax(t)-b)\rangle\\
	&&  = \langle x(t)-x^*, \nabla f(x(t))\rangle + \langle Ax(t)-b,\lambda^*\rangle + \sigma\|Ax(t)-b\|^2\\
	&&\geq f(x(t))-f(x^*)+ \langle Ax(t)-b,\lambda^*\rangle + \sigma\|Ax(t)-b\|^2\\
	&& = \mathcal{L}^{\sigma}(x(t),\lambda^*)-\mathcal{L}^{\sigma}(x^*,\lambda^*)+ \frac{\sigma}{2}\|Ax(t)-b\|^2.
\end{eqnarray*}
\end{small}
This together with \eqref{eq_th1_3} and \eqref{eq_th1_4} implies
\begin{eqnarray}\label{eq_th1_6}
	\dot{\mathcal{E}}(t) &=& \dot{\mathcal{E}}_0(t)+\dot{\mathcal{E}}_1(t)\nonumber\\
	&\leq& (\dot{\beta}(t)-\frac{1}{\delta}\beta(t))(\mathcal{L}^{\sigma}(x(t),\lambda^*)-\mathcal{L}^{\sigma}(x^*,\lambda^*))\nonumber\\
	&&-\frac{\sigma\beta(t)}{2\delta}\|Ax(t)-b\|^2+\frac{1-\delta\gamma}{\delta}\|\dot{x}(t)\|^2.
\end{eqnarray}
By \eqref{eq_Larg}, $\mathcal{L}^{\sigma}(x(t),\lambda^*)-\mathcal{L}^{\sigma}(x^*,\lambda^*)\geq 0$. From $ \eqref{eq_th1_beta_t}$, we have $\dot{\beta}(t)-\frac{1}{\delta}\beta(t))\leq 0$ and $\gamma\delta>1$.  Then ${\mathcal{E}}(t)\geq 0$,  and from \eqref{eq_th1_6} we get
\[\dot{\mathcal{E}}(t)\leq 0, \quad \forall t\in[t_0,T).\]
 As a consequence,   $\mathcal{E}(t)$  is nonincreasing on $[t_0, T)$,  and so
\begin{equation}\label{eq_th1_7}
	0\le \mathcal{E}(t)\leq  \mathcal{E}(t_0),\quad \forall t\in[ t_0,T).
\end{equation}
This together with \eqref{eq_th1_1} implies
\begin{equation*}\label{eq_th1_8}
	\frac{1}{2}\|\frac{1}{\delta}(x(t)-x^*)+\dot{x}(t)\|^2+\frac{\delta\gamma-1}{2\delta^2}\|x(t)-x^*\|^2\leq  \mathcal{E}(t_0)
\end{equation*}
for any $t\in[ t_0,T)$. Since $\delta\gamma-1>0$, we get
\begin{equation*}\label{eq_th1_9}
		\sup_{t\in[ t_0, T)} \|x(t)-x^*\|\leq  \sqrt{\frac{2\delta^2\mathcal{E}(t_0)}{\delta\gamma-1}} ,\quad \forall t\in[ t_0,T).
\end{equation*}
Using the triangle inequality, we get
 \begin{small}
 \begin{eqnarray*}\label{eq_th1_10}
 &&	\sup_{t\in[ t_0, T)}\|\dot{x}(t)\|\\
 	&&\leq\sup_{t\in[ t_0,T)}\|\frac{1}{\delta}(x(t)-x^*)+\dot{x}(t)\|+ \sup_{t\in[ t_0,T)}\frac{1}{\delta}\|x(t)-x^*\|\\	
 	&&\leq \sqrt{2\mathcal{E}(t_0)} + \sup_{t\in[ t_0,T)}\frac{1}{\delta}\|x(t)-x^*\|\\
 	&&< +\infty.
 \end{eqnarray*}
\end{small}
So $(x(t),\dot{x}(t))$ is bounded on $[t_0,T)$. By \eqref{eq_th1_1}, \eqref{eq_th1_7} and similar arguments, we have
\[\frac{1}{2\delta}\|\lambda(t)-\lambda^*\|^2\leq \mathcal{E}_1(t)\leq\mathcal{E}(t)\leq  \mathcal{E}(t_0),\quad \forall t\in[ t_0,T).  \]
Thus, $\lambda(t)$ is bounded on $[t_0, T)$.

Assume  on the contrary that $T<+\infty$.  Since  $(x(t), \lambda(t),\dot{x}(t))$ is bounded on $[t_0,T)$, by \eqref{dy_unpertu} and the assumptions we know that $(\ddot{x}(t),\dot{\lambda}(t))$ is bounded on $[t_0,T)$. It ensues that $(x(t), \lambda(t),\dot{x}(t))$  has a limit at $t = T$, and therefore can be continued. This arrives at a contradiction. Thus $T=+\infty$. 
By the above arguments, we can get the boundedness of $(x(t),\lambda(t))$  on $[t_0,+\infty)$.

Next, we show the results $(i)-(iii)$.
By integrating \eqref{eq_th1_6} on $[t_0,+\infty)$ and using \eqref{eq_th1_7}, we obtain
\begin{small}
\begin{eqnarray*}
	&&-{\mathcal{E}}(t_0)\leq\int^{\infty}_{t_0}(\dot{\beta}(t)-\frac{1}{\delta}\beta(t))(\mathcal{L}^{\sigma}(x(t),\lambda^*)-\mathcal{L}^{\sigma}(x^*,\lambda^*))dt\nonumber \\
	&&\quad- \int^{\infty}_{t_0}\frac{\sigma\beta(t)}{2\delta}\|Ax(t)-b\|^2dt+ \int^{\infty}_{t_0}\frac{1-\delta\gamma}{\delta}\|\dot{x}(t)\|^2dt\\
	&& \quad\leq 0.\nonumber
\end{eqnarray*}
\end{small}
It yields
\begin{eqnarray}\label{inter_abc}
	 &&\int^{\infty}_{t_0}(\frac{1}{\delta}\beta(t)-\dot{\beta}(t))(\mathcal{L}^{\sigma}(x(t),\lambda^*)-\mathcal{L}^{\sigma}(x^*,\lambda^*))dt\nonumber \\
	&&\quad+ \int^{\infty}_{t_0}\frac{\sigma\beta(t)}{2\delta}\|Ax(t)-b\|^2dt+ \frac{\delta\gamma-1}{\delta}\int^{\infty}_{t_0}\|\dot{x}(t)\|^2dt\nonumber\\
	&& \leq {\mathcal{E}}(t_0)\\
	&& <+\infty.\nonumber
\end{eqnarray}
Since $\frac{1}{\delta}\beta(t)-\dot{\beta}(t)\geq 0$ and $\delta\gamma-1> 0$,  each individual integral of \eqref{inter_abc} is nonnegative. It follows from  \eqref{inter_abc} that
\begin{equation*}\label{eq_th1_13}
	\int^{+\infty}_{t_0}(\frac{1}{\delta}\beta(t)-\dot{\beta}(t))(\mathcal{L}^{\sigma}(x(t),\lambda^*)-\mathcal{L}^{\sigma}(x^*,\lambda^*)) dt <+\infty,
\end{equation*}
\[ \int^{+\infty}_{t_0}\beta(t)\|Ax(t)-b\|^2 dt <+\infty\]
and
\[ \int^{+\infty}_{t_0}\|\dot{x}(t)\|^2 dt <+\infty. \]
If $\lim_{t\to+\infty}\beta(t)=+\infty$, by the definition of $\mathcal{E}(t)$ and \eqref{eq_th1_7} we get
\begin{equation*}\label{fyp282}
	\mathcal{L}^{\sigma}(x(t),\lambda^*)-\mathcal{L}^{\sigma}(x^*,\lambda^*) =\mathcal{O}\left(\frac{1}{\beta(t)}\right).	
\end{equation*}
To complete the proof, we  need to prove  $\mathcal{L}^{\sigma}(x^*,\lambda^*)=\mathcal{L}^{\sigma}(x^*,\lambda(t))$.

Since  $Ax^*=b$, it follows that
\begin{eqnarray*}
	&&\mathcal{L}^{\sigma}(x^*,\lambda(t)) \\
	&&= f(x^*)+\langle \lambda(t), Ax^*-b\rangle +\frac{\sigma}{2}\|Ax^*-b\|^2\\
&&=f(x^*)+\langle \lambda^*, Ax^*-b\rangle +\frac{\sigma}{2}\|Ax^*-b\|^2\\
	&& = \mathcal{L}^{\sigma}(x^*,\lambda^*).
 \end{eqnarray*}
\end{proof}

\begin{remark}
If the scaling condition \eqref{eq_th1_beta_t} is replaced with the   stronger one
$$\dot{\beta}(t)\leq(1-\kappa)\frac{1}{\delta}\beta(t),\qquad  \frac{1}{\delta}< \gamma $$
for some  $\kappa\in(0, 1]$,  then  by  Theorem \ref{th_dy_unpertu} $(i)$ we have
$$\int^{+\infty}_{t_0}\beta(t)(\mathcal{L}^{\sigma}(x(t),\lambda^*)-\mathcal{L}^{\sigma}(x^*,\lambda(t))) dt <+\infty.$$
It is worth mentioning that  Balhag et al. \cite{Balhag2020} discussed the convergence rate analysis of  the dynamical system \eqref{dy_HBFT} by  proving  $f(x(t))-\min f =\mathcal{O}(\frac{1}{\beta(t)})$ in \cite[Theorem 2.1]{Balhag2020} under the condition $\dot{\beta}(t)\leq \gamma\beta(t)$,
and
$$\int^{+\infty}_{t_0}\beta(t)(f(x(t))-\min f) dt <+\infty$$ in \cite[Proposition 2.3]{Balhag2020}
under the condition $\dot{\beta}(t)\leq(1-\kappa)\gamma\beta(t)$ for the unconstrained optimization problem \eqref{min_fun}. In this sense, Theorem \ref{th_dy_unpertu}  extends \cite[Theorem 2.1, Proposition 2.3]{Balhag2020} from the the unconstrained optimization problem \eqref{min_fun} to the linear equality constrained problem \eqref{question}.
\end{remark}

\begin{remark}
	As shown in Theorem \ref{th_dy_unpertu}, the time scaling coefficient $\beta(t)$ plays a crucial role in obtaining  convergence rates. The importance of the scaling  has been recognized in the design of accelerated dynamical systems and algorithms for convex optimization problems, see \cite{WibisonoWJ2016,Wilson2016,AttouchCRF2019}.
\end{remark}

By Theorem \ref{th_dy_unpertu}, we have the following result.
\begin{lemma}\label{cor_hx1}
	Suppose the assumptions of Theorem \ref{th_dy_unpertu} hold. Then
\begin{footnotesize}
\[
		\sup_{t\geq t_0} \left\|\beta(t)(Ax(t)-b)+\int^{t}_{t_0}(\frac{1}{\delta}\beta(s)- \dot{\beta}(s))(Ax(s)-b)ds\right\|
		<+\infty.
\]
\end{footnotesize}
\end{lemma}

\begin{proof}
By the second equation of \eqref{dy_unpertu} and applying the partial integration, we have
	\begin{eqnarray}\label{hx1}
		&& \lambda(t)-\lambda(t_0) =   \int^t_{t_0}\dot{\lambda}(s)ds\nonumber
\\
		 &&=\int^t_{t_0}  \beta(s)(A(x(s)+\delta \dot{x}(s))-b)ds\\
		 &&=  \int^t_{t_0}\beta(s)(Ax(s)-b)ds+\int^t_{t_0}  \delta\beta(s)d(A{x}(s)-b)\nonumber
\\
		&&= {\delta}\beta(t)(Ax(t)-b)-{\delta}\beta(t_0)(Ax(t_0)-b)\nonumber\\
	&&\quad+\delta\int^{t}_{t_0}(\frac{1}{\delta}\beta(s)-\dot{\beta}(s))(Ax(s)-b)ds,\nonumber
	\end{eqnarray}
where the third equality follows from the fact
$$A\dot{x}(s) = \frac{d(Ax(s)-b)}{ds}.$$ 
By Theorem \ref{th_dy_unpertu}, we know that ${\lambda}(t)$ is  bounded and
\[\sup_{t\geq t_0} \left\|\lambda(t)-\lambda(t_0)+{\delta}\beta(t_0)(Ax(t_0)-b)\right\|<+\infty.\]
 This together with \eqref{hx1} yields the desired result.
\end{proof}

Based on Lemma \ref{cor_hx1}, we can prove the  following $\mathcal{O}(1/t)$ ergodic convergence rate result.
\begin{theorem}\label{re_th1_3}
 Assume that $\nabla f$ is locally Lipschitz continuous, $\beta(t)\equiv \beta>0$, and $\gamma\delta>1$.  Let $(x(t),\lambda(t))$ be a global solution of the dynamical system \eqref{dy_unpertu} and $(x^*,\lambda^*)\in\Omega$. Then
\[|f(\bar{x}(t))-f(x^*)|=\mathcal{O}\left(\frac{1}{t}\right),\qquad \|A\bar{x}(t)-b\|= \mathcal{O}\left(\frac{1}{t}\right),\]
where $\bar{x}(t)=\frac{\int^t_{t_0}x(s)ds}{t-t_0}$.
\end{theorem}

\begin{proof}
 Since  $\beta(t)\equiv \beta>0$ and   $\gamma\delta>1$, \eqref{eq_th1_beta_t} is automatically  satisfied. It follows from Theorem \ref{th_dy_unpertu}$(i)$ and $\sigma>0$ that
\[\int^{+\infty}_{t_0}\mathcal{L}(x(t),\lambda^*)-\mathcal{L}(x^*,\lambda(t))dt<+\infty.\]
Notice that $\mathcal{L}(x^*,\lambda(t))=\mathcal{L}(x^*,\lambda^*)$. By the convexity of $ \mathcal{L}(\cdot,\lambda^*)$, we get
\begin{eqnarray*}
	&&\mathcal{L}(\bar{x}(t),\lambda^*)-\mathcal{L}(x^*,\lambda^*) \\
	&&\qquad\leq \frac{1}{t-t_0}\int^{t}_{t_0}\mathcal{L}(x(s),\lambda^*)-\mathcal{L}(x^*,\lambda^*)ds\\
	&&\qquad\leq \frac{1}{t-t_0}\int^{+\infty}_{t_0}\mathcal{L}(x(s),\lambda^*)-\mathcal{L}(x^*,\lambda(t))ds.
\end{eqnarray*}
This yields 
\begin{equation}\label{114_1}
	\mathcal{L}(\bar{x}(t),\lambda^*)-\mathcal{L}(x^*,\lambda^*)=\mathcal{O}(1/t).
\end{equation}
It follows from Lemma \ref{cor_hx1} that
\begin{equation}\label{hx_420_1}
	 \sup_{t\geq t_0} \beta\left\|Ax(t)-b+\frac{1}{\delta}\int^{t}_{t_0}(Ax(s)-b)ds\right\|<+\infty.
\end{equation}
By  Theorem \ref{th_dy_unpertu}, $x(t)$ is bounded, and then $\sup_{t\geq t_0} \|Ax(t)-b\|<+\infty$. Using  the triangle inequality, from \eqref{hx_420_1} we get
\[\sup_{t\geq t_0}\left\|\int^{t}_{t_0} (Ax(s)-b)ds\right\|<+\infty.\]
Since
\[\|A\bar{x}(t)-b\| = \frac{1}{t-t_0}\left\|\int^{t}_{t_0} (Ax(s)-b)ds\right\|,\]
we get 
\begin{equation}\label{114_2}
	\|A\bar{x}(t)-b\|= \mathcal{O}(1/t).
\end{equation}
Since $(x^*,\lambda^*)\in\Omega$, by the triangle inequality we have
\[ |f(\bar{x}(t))-f(x^*)|\leq\mathcal{L}(\bar{x}(t),\lambda^*)-\mathcal{L}(x^*,\lambda^*)+  \|\lambda^*\|\|A\bar{x}(t)-b\|. \]
This together with \eqref{114_1} and \eqref{114_2} implies
 \[ |f(\bar{x}(t))-f(x^*)|=\mathcal{O}(1/t).\]
\end{proof}

Remark that  $\mathcal{O}(1/t)$ ergodic convergence rate results on the heavy ball dynamical system and algorithm can be found in \cite{Ghadimi2015}.
In Theorem \ref{th_dy_unpertu},  we  can only obtain an $\mathcal{O}(1/\sqrt{\beta(t)})$   convergence rate of the objection function in the case $\lim_{t\to+\infty}\beta(t)=+\infty$  by using
\[|f({x}(t))-f(x^*)|\leq \mathcal{L}({x}(t),\lambda^*)-\mathcal{L}(x^*,\lambda(t)) + \|\lambda^*\|\|A{x}(t)-b\|.\]
Next, we shall derive an improved convergence rate when the  scaling coefficient satisfies  $\dot{\beta}(t)=\frac{1}{\delta}\beta(t)$. In this case, it is easy to verify that
\[\beta(t)=\mu e^{{t}/{\delta}}\quad \text{ with }\quad  \mu=e^{-{t_0}/{\delta}}.\]

\begin{theorem}\label{th_thex1}
	Assume that $\beta(t)=\mu e^{{t}/{\delta}}$ with $\mu>0$ and $\gamma\delta>1$. Let $(x(t),\lambda(t))$ be a global solution of the dynamical system \eqref{dy_unpertu} and  $(x^*,\lambda^*)\in\Omega$. Then
\[ |f(x(t))-f(x^*)| =\mathcal{O}\left(\frac{1}{ e^{{t}/{\delta}}}\right),\quad \|Ax(t)-b\| =\mathcal{O}\left(\frac{1}{{e^{{t}/{\delta}}}}\right).\]
\end{theorem}

\begin{proof}
Since $\dot{\beta}(t)=\frac{1}{\delta}\beta(t)$, the assumptions of Theorem \ref{th_dy_unpertu}  hold. By Lemma \ref{cor_hx1}, 
	\[ \sup_{t\geq t_0} \left\|\beta(t)(Ax(t)-b)\right\|<+\infty,\]
which yields
\[ \|Ax(t)-b\|=\mathcal{O}\left(\frac{1}{\beta(t)}\right)=\mathcal{O}\left(\frac{1}{{e^{{t}/{\delta}}}}\right). \]
By Theorem \ref{th_dy_unpertu}  $(iii)$ and using the triangle inequality, we get
\begin{eqnarray*}
	&&|f({x}(t))-f(x^*)|\\
	&&\qquad \leq \mathcal{L}({x}(t),\lambda^*)-\mathcal{L}(x^*,\lambda(t)) + \|\lambda^*\|\|A{x}(t)-b\|\\
	&&\qquad=\mathcal{O}\left(\frac{1}{{e^{{t}/{\delta}}}}\right).
\end{eqnarray*}
\end{proof}

\begin{remark}
For the second-order dynamical system in \cite{AttouchFC2021}, under the same parameter setting as in Theorem \ref{th_thex1}, \cite[Proposition 1]{AttouchFC2021} showed the convergence rates of  the feasibility violation $\|Ax(t)-b\|$ and the objective function error  $|f(x(t))-f(x^*)|$ are both  $\mathcal{O}(\frac{1}{{e^{{t}/{2\delta}}}})$. As a comparison, we have shown in Theorem \ref{th_thex1}  that  both  the feasibility violation and the objective function error  associated with  the dynamical system \eqref{dy_unpertu}  achieve $\mathcal{O}(\frac{1}{{e^{{t}/{\delta}}}})$.
\end{remark}

\section{The perturbed case}
In this section, we consider the following perturbed version of the dynamical system \eqref{dy_unpertu}:
\begin{equation}\label{dy_pertu}
	\begin{cases}
		\ddot{x}(t)+\gamma\dot{x}(t)& = -\beta(t)(\nabla f(x(t))\\
		&\quad +A^T\lambda(t)+\sigma A^T(Ax(t)-b))+\epsilon(t),\\
		\dot{\lambda}(t) &= \beta(t)(A(x(t)+\delta\dot{x}(t))-b),
	\end{cases}
\end{equation}
where $\epsilon: [t_0,+\infty)\to\mathbb{R}^n$ can be interpreted as a kind of disturbance. Here, we adopt the  terminology ``perturbation'' used by Attouch et al. \cite{AttouchCPR2018, AttouchCRF2019}.

To avoid repeating the proof,  we take for granted the existence and uniqueness of a global solution of  \eqref{dy_pertu}.  We shall show that the  convergence  properties established in Theorem \ref{th_dy_unpertu} - Theorem \ref{th_thex1} are preserved for  \eqref{dy_pertu} as $\epsilon(t)$ decays rapidly.

\begin{lemma}\cite[Lemma A.5]{Brezis1973}\label{le:A1}
	Let $\nu:[t_0,T]\to [0,+\infty)$ be integrable and $M\geq 0$. Suppose that $\mu:[t_0,T]\to \mathbb{R} $ is continuous and
	\[ \frac{1}{2}\mu(t)^2\leq \frac{1}{2}M^2+\int^{t}_{t_0}\nu(s)\mu(s)ds\]
	for all $t\in[t_0,T]$. Then 
	\[ |\mu(t)|\leq M+\int^t_{t_0}\nu(s)ds, \quad \forall t\in[t_0,T].\]
\end{lemma}

\begin{theorem}\label{th_dy_pertu}
Assume that $\epsilon:[t_0,+\infty)\to\mathbb{R}^n$ is an  integrable function such that
\begin{equation*}
\int^{+\infty}_{t_0}\|\epsilon(t)\|dt<+\infty
\end{equation*} and the scaling condition \eqref{eq_th1_beta_t} holds. Let $(x(t), \lambda(t))$ be a global solution of the dynamical system \eqref{dy_pertu} and $(x^*,\lambda^*)\in\Omega$. Then the trajectory $(x(t),\lambda(t))$ is bounded on $[t_0,+\infty)$ and
\begin{itemize}
	\item [(i)]$\int^{+\infty}_{t_0}(\frac{1}{\delta}\beta(t)-\dot{\beta}(t))(\mathcal{L}^{\sigma}(x(t),\lambda^*)-\mathcal{L}^{\sigma}(x^*,\lambda(t))) dt <+\infty$.
	\item [(ii)] $\int^{+\infty}_{t_0}\beta(t)\|Ax(t)-b\|^2 dt <+\infty$,\quad  $\int^{+\infty}_{t_0}\|\dot{x}(t)\|^2dt <+\infty$.
	\item [(iii)] When $\beta(t)\equiv\beta>0$:
 \[|f(\bar{x}(t))-f(x^*)|=\mathcal{O}\left(\frac{1}{t}\right),\quad \|A\bar{x}(t)-b\|= \mathcal{O}\left(\frac{1}{t}\right),\]
 with $\bar{x}(t)=\frac{\int^t_{t_0}x(s)ds}{t-t_0}$.\\
\item [(iv)] 	When $ \lim_{t\to+\infty}\beta(t) = +\infty$:
	\[ \mathcal{L}(x(t),\lambda^*)-\mathcal{L}(x^*,\lambda(t)) =\mathcal{O}\left(\frac{1}{\beta(t)}\right),\]
	and
	\[\|Ax(t)-b\| =\mathcal{O}\left(\frac{1}{\sqrt{\beta(t)}}\right).\]
	\item [(v)] Moreover, when $\beta(t)=\mu e^{{t}/{\delta}}$ with $\mu>0$:
	 \[ |f(x(t))-f(x^*)| =\mathcal{O}\left(\frac{1}{ e^{{t}/{\delta}}}\right)\]
	 and
	 \[\|Ax(t)-b\| =\mathcal{O}\left(\frac{1}{{e^{{t}/{\delta}}}}\right).\]
\end{itemize}\end{theorem}
\begin{proof}
Define the energy  function  $\mathcal{E}^{\epsilon}:[t_0,+\infty)\to\mathbb{R}$ by
\begin{eqnarray}\label{eq_th2_1}
	\mathcal{E}^{\epsilon}(t) =\mathcal{E}(t)- \int^t_{t_0}\langle \frac{1}{\delta}(x(s)-x^*)+\dot{x}(s),\epsilon(s)\rangle ds,
\end{eqnarray}
where $\mathcal{E}(t)$ is defined as in \eqref{eq_th1_1}.
By similar arguments as in the proof of Theorem \ref{th_dy_unpertu}, we have
\begin{eqnarray}\label{eq_th2_2}
	&&\dot{\mathcal{E}}^{\epsilon}(t) \leq (\dot{\beta}(t)-\frac{1}{\delta}\beta(t))(\mathcal{L}^{\sigma}(x(t),\lambda^*)-\mathcal{L}^{\sigma}(x^*,\lambda(t)))\nonumber\\
	&&\quad-\frac{\sigma\beta(t)}{2\delta}\|Ax(t)-b\|^2-\frac{\delta\gamma-1}{\delta}\|\dot{x}(t)\|^2
\end{eqnarray}
for any $t\geq t_0$. By \eqref{eq_th2_2} and \eqref{eq_th1_beta_t},  $\mathcal{E}^{\epsilon}(t)$  is nonincreasing on $[t_0,+\infty)$ and then
\begin{equation*}\label{eq_th2_3}
	\mathcal{E}^{\epsilon}(t)\leq  \mathcal{E}^{\epsilon}(t_0),\quad \forall t\in[ t_0,+\infty).
\end{equation*}
This together with  \eqref{eq_th2_1} implies
\begin{equation}\label{eq_th2_4}
	\mathcal{E}(t)\leq  \mathcal{E}^{\epsilon}(t_0)+\int^t_{t_0}\langle \frac{1}{\delta}(x(s)-x^*)+\dot{x}(s),\epsilon(s)\rangle ds
	\end{equation}
for any $t\in[ t_0,+\infty)$. By the definition of $\mathcal{E}(t)$ and the Cauchy-Schwarz inequality, from \eqref{eq_th2_4} we have
\begin{eqnarray*}\label{eq_th2_5}
	&&\frac{1}{2}\|\frac{1}{\delta}(x(t)-x^*)+\dot{x}(t)\|^2\leq \mathcal{E}(t)\\
	&&\quad\leq |\mathcal{E}^{\epsilon}(t_0)|+\int^t_{t_0}\|\frac{1}{\delta}(x(s)-x^*)+\dot{x}(s)\|\cdot\|\epsilon(s)\| ds
\end{eqnarray*}
for any $t\geq t_0$.
Applying Lemma \ref{le:A1} with $\mu(t) =\|\frac{1}{\delta}(x(t)-x^*)+\dot{x}(t)\|$ to the above inequality, we get
\begin{equation*}\label{eq_th2_6}
  \|\frac{1}{\delta}(x(t)-x^*)+\dot{x}(t)\| \leq  \sqrt{2|\mathcal{E}^{\epsilon}(t_0)|}+\int^t_{t_0}\|\epsilon(s)\| ds
\end{equation*}
for any $t\geq t_0$.
This together with $\int^{+\infty}_{t_0}\|\epsilon(s)\| ds<+\infty$ yields
\begin{equation*}\label{eq_th2_7}
	 \sup_{t\geq t_0}\|\frac{1}{\delta}(x(t)-x^*)+\dot{x}(t)\| <  +\infty.
\end{equation*}
Since  $\mathcal{E}(t)\geq 0$ for all $t\geq t_0$, it follows from \eqref{eq_th2_1} and \eqref{eq_th2_4} that
\begin{eqnarray*}\label{eq_th2_8}
	\inf_{t\geq t_0 }\mathcal{E}^{\epsilon}(t)&\geq& -\sup_{t\geq t_0}\|\frac{1}{\delta}(x(t)-x^*)+\dot{x}(t)\|\int^{+\infty}_{t_0}\|\epsilon(s)\| ds \\
	&>& -\infty
\end{eqnarray*}
and
\begin{eqnarray*}\label{eq_th2_9}
	&&\sup_{t\geq t_0 }\mathcal{E}(t) \\
	&& \leq \mathcal{E}^{\epsilon}(t_0)+\sup_{t\geq t_0}\|\frac{1}{\delta}(x(t)-x^*)+\dot{x}(t)\|\int^{+\infty}_{t_0}\|\epsilon(s)\| ds \\
	&&< +\infty.
\end{eqnarray*}
This means that  $\mathcal{E}^{\epsilon}(t)$ and $\mathcal{E}(t)$ are both bounded on $[t_0,+\infty)$. It yields that $(x(t),\lambda(t))$ is bounded. Integrating \eqref{eq_th2_2} on $[t_0,+\infty)$, we  obtain $(i)$ and $(ii)$. The rest of the proof is similar to the one in Theorem \ref{th_dy_unpertu}, Theorem \ref{re_th1_3}, Theorem \ref{th_thex1},  and so we omit it.
\end{proof}

\section{Conclusions}
In this paper, we propose an inertial primal-dual  dynamical system with time scaling for the problem \eqref{question}, which includes a second-order ODE for the prime variable  and a first-order ODE for the dual variable.  Under a scaling condition, we prove that the proposed dynamical system and its perturbed version enjoy a  convergence property: $\mathcal{L}(x(t),\lambda^*)-\mathcal{L}(x^*,\lambda(t)) =\mathcal{O}(\frac{1}{\beta(t)})$.  In the case  $\beta(t)=\mu e^{{t}/{\delta}}$, we derive an improved exponential  convergence rate: $|f(x(t))-f(x^*)| =\mathcal{O}(\frac{1}{ e^{{t}/{\delta}}})$.  The  convergence results of our approaches is realized without strong convexity of the objective function.

\section*{Acknowledgments}
The authors would like to thank the reviewers and the editors for their helpful comments and suggestions, which improve the quality of this paper.

\end{document}